\documentclass{aims-ppn}
\usepackage{amsmath}
\usepackage{paralist}
\usepackage{graphics}
\usepackage{epsfig}
\usepackage{graphicx}\usepackage{epstopdf}
\usepackage{mathrsfs}
\usepackage[colorlinks=true]{hyperref}
\hypersetup{urlcolor=blue, citecolor=red}

1

\newtheorem{theorem}{Theorem}
\newtheorem{corollary}[theorem]{Corollary}
\newtheorem{lemma}[theorem]{Lemma}
\newtheorem{proposition}[theorem]{Proposition}

\theoremstyle{definition}

\newcommand{\msc}[1]{\href{https://mathscinet.ams.org/mathscinet/search/mscbrowse.html?sk=default&sk=#1&submit=Chercher}{#1}}
\newcommand{\R}{{\mathbb R}}

\renewcommand{\S}{{\mathbb S}}
\newcommand{\be}[1]{\begin{equation}\label{#1}}
\newcommand{\ee}{\end{equation}}
\renewcommand{\(}{\left(}
\renewcommand{\)}{\right)}
\newcommand{\dx}{\mathrm{d}x}
\newcommand{\dt}{\mathrm{d}t}
\newcommand{\rd}{\mathrm{d}}
\newcommand{\ird}[1]{\int_{\R^d}{#1}\,\dx}
\newcommand{\nrm}[2]{\left\|{#1}\right\|_{#2}}
\newcommand{\DD}{\mathsf D_\alpha\kern0.5pt}
\newcommand{\DDstar}{\mathsf D_\alpha^*\kern0.5pt}
\renewcommand{\L}{{\mathcal L_\alpha\,}}
\newcommand{\nrmL}[2]{\|{#1}\|_{#2}}
\newcommand{\C}{{\mathsf C}}
\newcommand{\irdsph}[3]{\int_0^\infty\kern-5pt\int_{\S^{d-1}}#2\;{#1}^{\kern1pt #3}\,\frac{d{#1}}{{#1}}\;d\omega}
\newcommand{\D}[1]{\mathsf D_\alpha\kern1pt#1}
\newcommand{\muscal}{\lambda}
\newcommand{\mB}{\mathcal B}
\newcommand{\Mstar}{\mathcal M}

\newcommand{\LL}{\mathrm{L}}
\renewcommand{\a}{\mathrm{a}}
\renewcommand{\c}{\mathscr{C}_\star}

\title[Stability and improved decay rates]{Constructive stability results in interpolation inequalities and explicit improvements of decay rates of fast diffusion equations} 

\author[M.~Bonforte, J.~Dolbeault, B.~Nazaret and N.~Simonov]{}

\email{matteo.bonforte@uam.es}
\email{dolbeaul@ceremade.dauphine.fr}
\email{bruno.nazaret@univ-paris1.fr}
\email{simonov.nikita@gmail.com}

\thanks{$^*$ Corresponding author: J.~Dolbeault}

\begin{document}

\maketitle
\thispagestyle{empty}

\centerline{\scshape Matteo Bonforte}
\smallskip
{\footnotesize\centerline{Departamento de Matem\'{a}ticas}
\centerline{Universidad Aut\'{o}noma de Madrid, and ICMAT}
\centerline{Campus de Cantoblanco, 28049 Madrid, Spain}}

\medskip

\centerline{\scshape Jean Dolbeault}
\smallskip
{\footnotesize\centerline{Ceremade, UMR CNRS n$^\circ$~7534}
\centerline{Universit\'e Paris-Dau\-phine, PSL Research University}
\centerline{Place de Lattre de Tassigny, 75775 Paris Cedex~16, France}}

\medskip

\centerline{\scshape Bruno Nazaret}
\smallskip
{\footnotesize\centerline{SAMM (EA 4543), FP2M (FR CNRS 2036)}
\centerline{Universit\'e Paris 1}
\centerline{90, rue de Tolbiac, 75634 Paris Cedex~13}}

\medskip

\centerline{\scshape Nikita Simonov}
\smallskip
{\footnotesize\centerline{LaMME, UMR CNRS n$^\circ$~8071}
\centerline{Universit\'e d'\'Evry Val d’Essonne}
\centerline{23, boulevard de France, 91037 \'Evry, France}}

\bigskip


\begin{abstract} We provide a scheme of a recent stability result for a family of Gagliardo-Nirenberg-Sobolev (GNS) inequalities, which is equivalent to an improved entropy -- entropy production inequality associated with an appropriate fast diffusion equation (FDE) written in self-similar variables. This result can be rephrased as an improved decay rate of the entropy of the solution of (FDE) for well prepared initial data. There is a family of Caffarelli-Kohn-Nirenberg (CKN) inequalities which has a very similar structure. When the exponents are in a range for which the optimal functions for (CKN) are radially symmetric, we investigate how the methods for (GNS) can be extended to (CKN). In particular, we prove that the solutions of the evolution equation associated to (CKN) also satisfy an improved decay rate of the entropy, after an explicit delay. However, the improved rate is obtained without assuming that initial data are well prepared, which is a major difference with the (GNS) case.
\end{abstract}

\begin{center}\emph{It is with great pleasure that we dedicate this paper to Juan Luis V\'azquez\\ on the occasion of his 75$^{th}$ birthday.} \end{center}

\bigskip

\begin{center}\parbox{10.2cm}{

\medskip


}\end{center}\newpage


\section{Introduction and main results}

Let us start with some Gagliardo-Nirenberg-Sobolev inequalities (without weights) and related flow issues in Section~\ref{Sec:introGNS} before extending the results to Caffarelli-Kohn-Nirenberg inequalities (with weights)  and related flows in Section~\ref{Sec:introCKN}.

\medskip\subsection{Gagliardo-Nirenberg-Sobolev inequalities and related flows}\label{Sec:introGNS}~

The \emph{Gagliardo-Nirenberg-Sobolev inequality}
\begin{equation}\label{GNS}
\nrm{\nabla f}2^\theta\,\nrm f{p+1}^{1-\theta}\ge\mathcal C_{\mathrm{GNS}}\,\nrm f{2\,p}
\end{equation}
holds on the space of the functions $f\in\mathrm L^{p+1}(\mathbb R^d)$ with $\nabla f\in\mathrm L^2(\mathbb R^d)$, with exponents given by
\[
\theta=\tfrac{d\,(p-1)}{(d+2-p\,(d-2))\,p}\,,\quad p\in(1,+\infty)\;\mbox{if}\;d=1\;\mbox{or}\;2\,,\quad p\in(1,p_\star]\;\mbox{if}\;d\ge3\,,\quad p_\star=\tfrac d{d-2}\,.
\]
According to~\cite{Gunson91,DelPino2002}, equality in~\eqref{GNS} is achieved if and only if $f$ is equal to
\be{Aubin-TalentiGNS}
\mathsf g(x)=\big(1+|x|^2\big)^{-\frac1{p-1}}\quad\forall\,x\in\R^d\,,
\ee
up to a multiplication by a constant, a translation and a scaling. We denote by~$\mathfrak M$ the manifold of optimal functions for~\eqref{GNS}. This inequality has a number of interesting limit cases: Sobolev's inequality if $d\ge3$ and $p=p_\star$, the Euclidean Onofri inequality if $d=2$ in the limit as $p\to+\infty$, and the scale invariant Euclidean logarithmic Sobolev inequality in the limit as $p\to1_+$. Let us define the \emph{deficit functional}
\[\label{deficit}
\delta[f]:=(p-1)^2\,\nrm{\nabla f}2^2+4\,\tfrac{d-p\,(d-2)}{p+1}\,\nrm f{p+1}^{p+1}-\mathcal K_{\mathrm{GNS}}\,\nrm f{2\,p}^{2\,p\,\chi}
\]
with $\chi=\frac{d+2-p\,(d-2)}{d-p\,(d-4)}$ and $\mathcal K_{\mathrm{GNS}}$ chosen so that $\delta[\mathsf g]=0$. Up to a scaling, the inequality $\delta[f]\ge0$ is equivalent to~\eqref{GNS}, as was shown in~\cite{DelPino2002}.

In the critical case $p=p_\star$, $d\ge3$, optimal functions in~\eqref{GNS} are known as the \emph{Aubin-Talenti functions} and this result goes back to~\cite{Aubin-76,Talenti-76,rodemich1966sobolev}. Later, in~\cite{MR790771}, H.~Brezis and E.H.~Lieb asked the next natural question: which distance to $\mathfrak M$ is controlled by the deficit $\delta$ ? Soon after, an answer was given in~\cite{MR1124290} by G.~Bianchi and H.~Egnell: there is a positive constant~$\mathcal C_{\mathrm{BE}}$ such that
\[
\tfrac1{(p-1)^2}\,\delta[f]=\nrmL{\nabla f}2^2-\mathsf S_d\,\nrmL f{2^*}^2\ge\mathcal C_{\mathrm{BE}}\,\inf_{\varphi\in\mathfrak M}\,\nrmL{\nabla f-\nabla\varphi}2^2\,,
\]
where $\mathsf S_d$ is the optimal constant in Sobolev's inequality. This striking result had anyway an important drawback: $\mathcal C_{\mathrm{BE}}$ is obtained by a \emph{non-constructive} method. Various extensions and improvements as, \emph{e.g.}, in~\cite{Cianchi2009,figalli2020sharp} have been obtained, as well as similar results for~\eqref{GNS} in~\cite{Carlen2013,MR3695890}, in the subcritical range: we refer to~\cite{BDNS2021} for a review of the literature. New \emph{constructive stability results} were recently obtained:
\begin{theorem}\label{Thm:BDNS}\cite[Corollary~5.4 and Theorem~6.1]{BDNS2021}\label{Thm:stabilityBDNS} Let $d\ge3$ and $p\in(1,p_\star]$, or $d=1$,~$2$ and $p\in(1,+\infty)$. For any $f\in\mathrm L^{2p}(\mathbb R^d)$ with $\nabla f\in\mathrm L^2(\mathbb R^d)$ and
\[
A:=\sup_{r>0}r^\frac{d-p\,(d-4)}{p-1}\int_{|x|>r}|f|^{2p}\,\dx<\infty\,,
\]
we have the estimate
\be{Ineq:StabFisher}
\delta[f]\ge\mathcal C\,\inf_{\varphi\in\mathfrak M}\ird{\big|(p-1)\,\nabla f+f^p\,\nabla\varphi^{1-p}\big|^2}
\ee
for some explicit $\mathcal C>0$ which depends only on $d$, $p$, $\nrm f{2p}$ and~$A$. 
\end{theorem}
\noindent As a function of $f$ or, to be precise, as a function of the mass $\nrm f{2p}$ and $A$, $\mathcal C$ takes positive values on $\mathfrak M$ even if it is not uniformly bounded away from $0$ on $\mathfrak M$. The distance to $\mathfrak M$ is measured by a \emph{Fisher information} functional and the strategy of the proof involves \emph{entropy} methods. Inequality~\eqref{Ineq:StabFisher} is equivalent to an improved \emph{entropy -- entropy production inequality}, which relates the Fisher information with a \emph{relative entropy}, or \emph{free energy}, defined for $m\in[m_1,1)$ with $m_1:=1-1/d$ by
\be{Entropy}
\mathcal F[v]:=\frac1{m-1}\ird{\(v^m-\mathcal B^m-m\,\mathcal B^{m-1}\,(v-\mathcal B)\)}
\ee
where
\be{BarenblattGNS}
\mathcal B(x):=\big(1+|x|^2\big)^\frac1{m-1}\quad\forall\,x\in\R^d\,.
\ee
The functional $\mathcal F$ enters in the study of nonlinear evolution equations as follows. Let us consider the \emph{fast diffusion equation}
\be{FDr}
\frac{\partial v}{\partial t}+\nabla\cdot\(v\,\nabla v^{m-1}\)=2\,\nabla\cdot(x\,v)\,,\quad v(t=0,\cdot)=v_0\,.
\ee
By a standard computation, a solution $v$ of~\eqref{FDr} is such that
\be{Fisher}
\frac \rd{\dt}\mathcal F[v(t,\cdot)]=-\,\mathcal I[v(t,\cdot)]\;\mbox{where}\;\mathcal I[v]:=\frac m{1-m}\ird{v\,\big|\nabla v^{m-1}-\nabla \mathcal{B}^{m-1}\big|^2}
\ee
where $\mathcal I[v]$ is the \emph{relative Fisher information} with respect to $\mathcal B$. The exponents $m$ in~\eqref{FDr} and $p$ in~\eqref{GNS} are related by the condition $p=1/(2\,m-1)$. In practice, for $m>m_1$ (subcritical range), the result of Theorem~\ref{Thm:BDNS} can be rephrased as a result on \emph{decay rates} for the solutions of~\eqref{FDr}.
\begin{corollary}\label{Cor:Improvedrate}\cite[Corollary~5.2]{BDNS2021} Let $d\ge1$ and $m\in(m_1,1)$. If $v$ is a solution of~\eqref{FDr} with nonnegative initial datum $v_0\in\mathrm L^1(\R^d)$ such that $\ird{v_0}=\ird{\mathcal B}$, $\ird{x\,v_0}=0$ and
\[
 A[v_0]:=\sup_{r>0}r^{\frac2{1-m}-d}\int_{|x|>r}v_0\,\dx<\infty\,,
\]
then we have
\be{improved.rate.GNS}
\mathcal F[v(t,.)]\le\mathcal F[v_0]\,e^{-\,(4+\zeta)\,t}\quad\forall\,t\ge0
\ee
for some positive constant $\zeta$ which depends explicitly only on $m$, $d$, and~$A[v_0]$.
\end{corollary}
\noindent The integral $\ird{x\,v_0}$ is finite because $A[v_0]<+\infty$. Note that if the center of mass of $v_0$ is finite, then $\ird{x\,v(t,x)}=\ird{x\,v_0}\,e^{-2t}$ for any $t\ge0$.

Let us consider the \emph{optimized free energy functional}
\[
\mathcal F_\star[v]:=\inf_{B\in\mathfrak B}\frac1{m-1}\ird{\(v^m-B^m-m\,B^{m-1}\,(v-B)\)}
\]
where $\mathfrak B$ is the set of all \emph{Barenblatt} functions obtained from $\mathcal B$ using a multiplication by a constant, translations and scalings.
\begin{corollary}\label{Cor:Improvedrate2} Let $m\in[m_1,1)$ if $d\ge2$, $m\in(1/2,1)$ if $d=1$ and consider $\zeta$ as in~Corollary~\ref{Cor:Improvedrate}. If $v$ is a solution of~\eqref{FDr} with nonnegative initial datum $v_0\in\mathrm L^1(\R^d)$ such that $A[v_0]$ is finite, then
\be{improved.rate.GNS2}
\mathcal F_\star[v(t,.)]\le\mathcal F_\star[v_0]\,e^{-\,(4+\zeta)\,t}\quad\forall\,t\ge0\,.
\ee
\end{corollary}
\noindent In the subcritical range, this result is a straightforward consequence of~Corollary~\ref{Cor:Improvedrate}, but it is new in the critical case $m=m_1$ corresponding to $p=p_\star$.

\medskip The \emph{entropy -- entropy production inequality} relates rates of convergence for the solutions of~\eqref{FDr} with~\eqref{GNS}, but~\eqref{GNS} can also be invoked in the context of the standard \emph{fast diffusion equation}
\be{FD}
\frac{\partial u}{\partial t}=\Delta u^m\,,\quad u(t=0,\cdot)=u_0\,.
\ee
Assume that $m\in[m_1,1)$ if $d\ge2$ and $m\in(1/2,1)$ if $d=1$. Using the \emph{R\'enyi entropy powers} formalism, we learn  from~\cite[Lemma~2.1]{BDNS2021} that a solution~$u$ of~\eqref{FD} with initial datum $u_0\in\mathrm L^1_+\!\(\R^d,(1+|x|^2)\,\mathrm dx\)$ such that $u_0^m\in\mathrm L^1(\R^d)$ satisfies
\be{Ch2:GrowthEntropyGNS}
\ird{u^m(t,x)}\ge\(\(\ird{u_0^m}\)^\frac{m-m_c}{1-m}+\tfrac{(1-m)\,C_0}{m-m_c}\,t\)^\frac{1-m}{m-m_c}\quad\forall\,t\ge0\,,
\ee
for some constant $C_0$ which explicitly involves $\mathcal C_{\mathrm{GNS}}$ and where $m_c:=(d-2)/d$. Equality in~\eqref{Ch2:GrowthEntropyGNS} holds for any $t\ge0$ if and only if $u_0\in\mathfrak B$. At $t=0$, \eqref{Ch2:GrowthEntropyGNS} is an equality for any $u_0$ and we can recover the Gagliardo-Nirenberg-Sobolev inequality~\eqref{GNS} written with the optimal constant by differentiating with respect to $t$ the growth estimate~\eqref{Ch2:GrowthEntropyGNS} at $t=0$. A more readable estimate is obtained by considering the \emph{optimized free energy functional} $\mathcal F_\star$ applied to~\eqref{FD}.
\begin{corollary}\label{Cor:Improvedrate3} Assume that $d\ge1$, $m\in[m_1,1)$ and consider a solution of~\eqref{FD} with initial datum $u_0\in\mathrm L^1_+\!\(\R^d,(1+|x|^2)\,\mathrm dx\)$ such that $u_0^m\in\mathrm L^1(\R^d)$. With $\kappa=\zeta/d$ and~$\zeta$ as in Corollary~\ref{Cor:Improvedrate2}, we have
\[
\mathcal F_\star[u(t,.)]\le\mathcal F_\star[u_0]\,\big(1+d\,(m-m_c)\,t\big)^{-\frac{m+\kappa}{m-m_c}}\quad\forall\,t\ge0\,.
\]
\end{corollary}
\noindent This new result is remarkable. While the best matching function \hbox{$B(t,\cdot)\in\mathfrak B$} is such that $\ird{B^m(t,x)}\sim t^{(1-m)/(m-m_c)}\to+\infty$ as $t\to+\infty$, according to~\eqref{Ch2:GrowthEntropyGNS}, it turns out that $\mathcal F[u(t,.)]$, which involves $\ird{\(u^m(t,x)-B^m(t,x)\)}$, decays to $0$ at an algebraic rate. As a consequence, we have $\lim_{t\to+\infty}\nrm{u(t,.)-B(t,.)}1=0$, with an explicit rate, by the Csisz\'ar-Kullback-Pinsker inequality. The case $\kappa=0$ is a consequence of~\eqref{GNS} and the improvement $\kappa>0$ is a consequence of Theorem~\ref{Thm:stabilityBDNS}.

\medskip\subsection{Caffarelli-Kohn-Nirenberg inequalities and related flows}\label{Sec:introCKN}

So far the results are simple consequences of the method of~\cite{BDNS2021}. We are now going to extend them to a larger class of inequalities. On $\R^d$ with $d\ge1$, let us consider the \emph{Caffarelli-Kohn-Nirenberg interpolation inequalities}
\be{CKN}
\nrm f{2p,\gamma}\le\C_{\beta,\gamma,p}\,\nrm{\nabla f}{2,\beta}^\theta\,\nrm f{p+1,\gamma}^{1-\theta}
\ee
with optimal constant $\C_{\beta,\gamma,p}$, parameters $\beta$, $\gamma$ and $p$ such that
\be{parameters}
\gamma-2<\beta<\frac{d-2}d\,\gamma\,,\quad\gamma\in(-\infty,d)\,,\quad p\in\(1,p_\star\right]\quad\mbox{with}\quad p_\star:=\frac{d-\gamma}{d-\beta-2}\,,
\ee
and an exponent
\[\label{theta}
\theta=\frac{(d-\gamma)\,(p-1)}{p\,\big(d+\beta+2-2\,\gamma-p\,(d-\beta-2)\big)}
\]
which is determined by the scaling invariance. This formula for $\theta$ extends to~\eqref{CKN} the expression for~\eqref{GNS}, which corresponds to the special case $\beta=\gamma=0$. Here $\mathrm L^{q,\gamma}(\R^d)$ and $\mathrm L^q(\R^d)$ respectively denote the spaces of all measurable functions~$f$ such that
\[
\nrm f{q,\gamma}:=\(\ird{|f|^q\,|x|^{-\gamma}}\)^{1/q}\quad\mbox{and}\quad\nrm fq:=\nrm f{q,0}
\]
are finite. Inequality~\eqref{CKN} holds in the space $\mathrm H^p_{\beta,\gamma}(\R^d)$ of functions $f\in\mathrm L^{p+1,\gamma}(\R^d)$ such that $\nabla f\in\mathrm L^{2,\beta}(\R^d)$, defined as the completion of the space $\mathcal D(\R^d\setminus\{0\})$ of the smooth functions on $\R^d$ with compact support in~$\R^d\setminus\{0\}$, with respect to the norm given by $f\mapsto(p_\star-p)\,\nrm f{p+1,\gamma}^2+\nrm{\nabla f}{2,\beta}^2$. Since the weights are locally integrable, these spaces can also be defined as the completion of the space $\mathcal D(\R^d)$. The limitation $p\le p_\star$ in~\eqref{parameters} amounts, for a given $p>1$ to a restriction to the admissible set of parameters $(\beta,\gamma)$, namely $\beta\ge d-2+(\gamma-d)/p$. On the other hand, we notice that
\[
\beta<\frac{d-2}d\,\gamma\quad\Longleftrightarrow\quad p_\star<\frac d{d-2}\,.
\]
Inequality~\eqref{CKN} belongs to a family of inequalities introduced by L.~Caffarelli, R.~Kohn and L.~Nirenberg in~\cite{Caffarelli1984} and also earlier by V.P.~Il'in in~\cite{Ilyin}. The range of admissible parameters $(\beta,\gamma)$ is limited by~\eqref{parameters} to a cone in the quadrant $\beta<d-2$ and $\gamma<d$ (see Fig.~\ref{Fig}), but the inequality also holds in a cone in the quadrant $\beta>d-2$ and $\gamma>d$ using the property of \emph{inversion symmetry}: see~\cite[Section~2.1]{MR3579563} for details.

A central issue in Caffarelli-Kohn-Nirenberg inequalities~\eqref{CKN} is to decide whether the equality case is achieved among radial functions or not when $d\ge2$. We summarize this alternative by \emph{symmetry} versus \emph{symmetry breaking}. Symmetry in~\eqref{CKN} means that the equality case is achieved by the (generalized) \emph{Aubin-Talenti type functions}
\be{Aubin-Talenti}
\mathsf g(x)=\big(1+|x|^\sigma\big)^{-\frac1{p-1}}\quad\forall\,x\in\R^d\,,\quad\mbox{with}\quad\sigma:=2+\beta-\gamma\,.
\ee
This definition of Aubin-Talenti type functions extends the one in~\eqref{Aubin-TalentiGNS}. In the critical case $p=p_\star$, $\theta=1$, V.~Felli and M.~Schneider proved in~\cite{Felli2003} that \emph{symmetry breaking} holds~if
\[
\gamma<0\quad\mbox{and}\quad\beta_{\rm FS}(\gamma)<\beta<\frac{d-2}d\,\gamma\,,
\]
where
\[
\beta_{\rm FS}(\gamma):=d-2-\sqrt{(\gamma-d)^2-4\,(d-1)}\,.
\]
Reciprocally, if
\[
\gamma<d\,,\quad\mbox{and}\quad\gamma-2<\beta<\frac{d-2}d\,\gamma\quad\mbox{and}\quad\beta\le\beta_{\rm FS}(\gamma)\,,
\]
then \emph{symmetry} holds according to~\cite{DEL-2015}. The results are exactly the same in the subcritical case $p\in\(1,p_\star\)$ as was shown in~\cite[Theorem~2]{MR3579563} and in~\cite[Theorem~1.1]{Dolbeault2017}. See Fig.~\ref{Fig}.
\addtolength{\textwidth}{30pt}\setlength{\unitlength}{1cm}
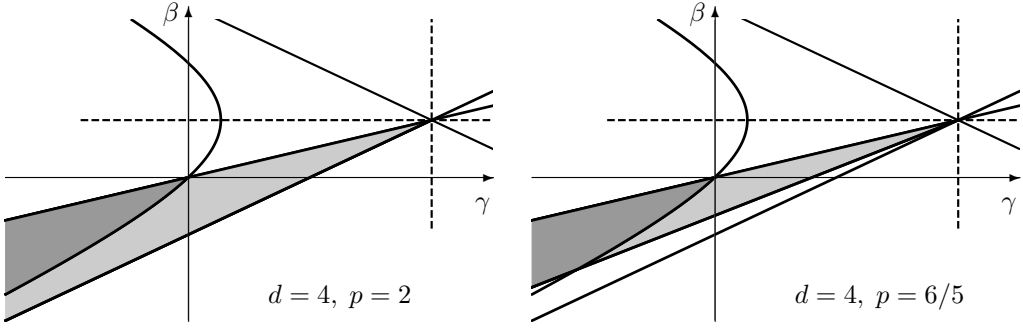
\begin{figure}[ht]
\begin{center}\hspace*{-25pt}
\begin{picture}(13.5,6){\includegraphics[width=6.5cm]{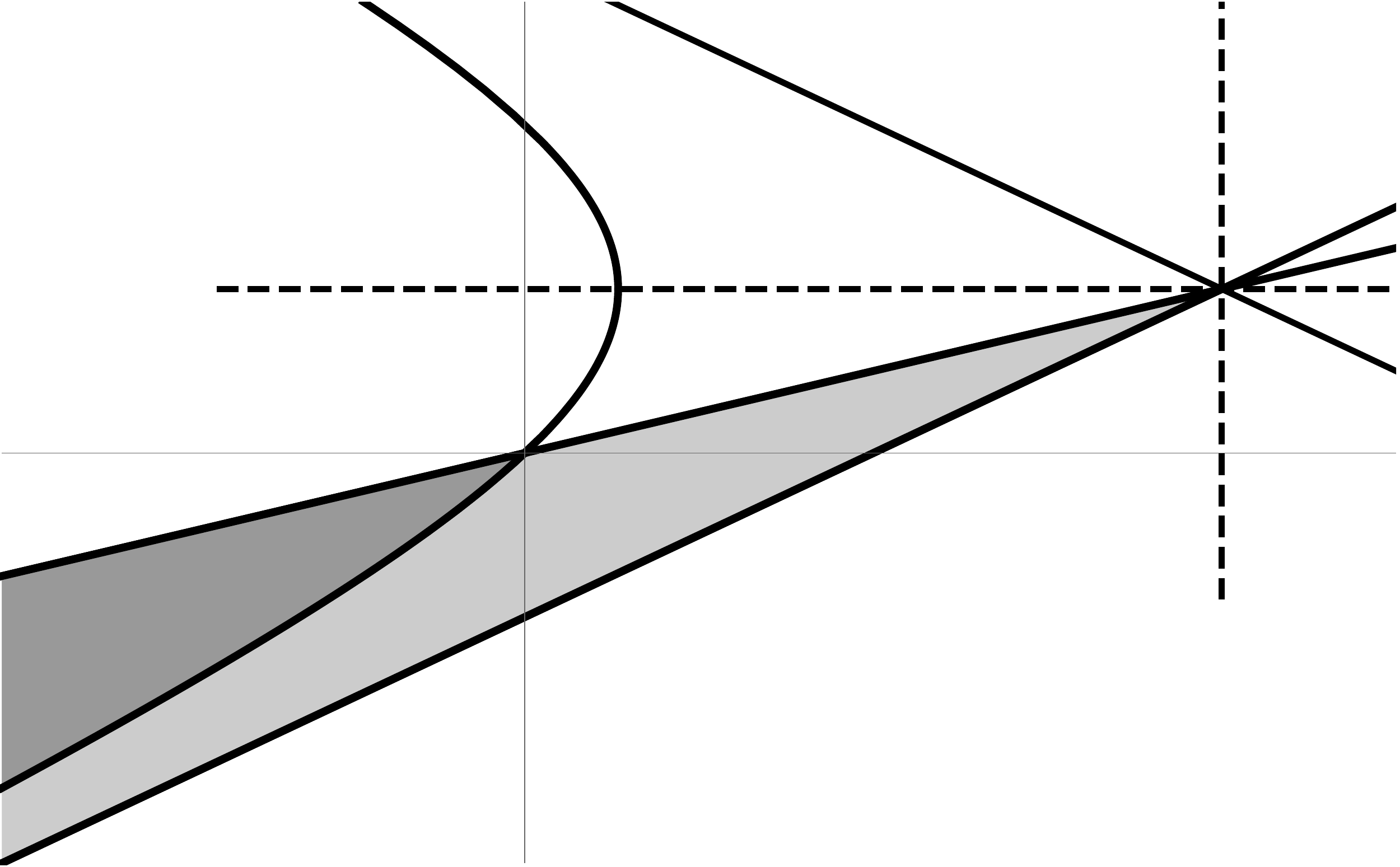}\hspace*{0.5cm}\includegraphics[width=6.5cm]{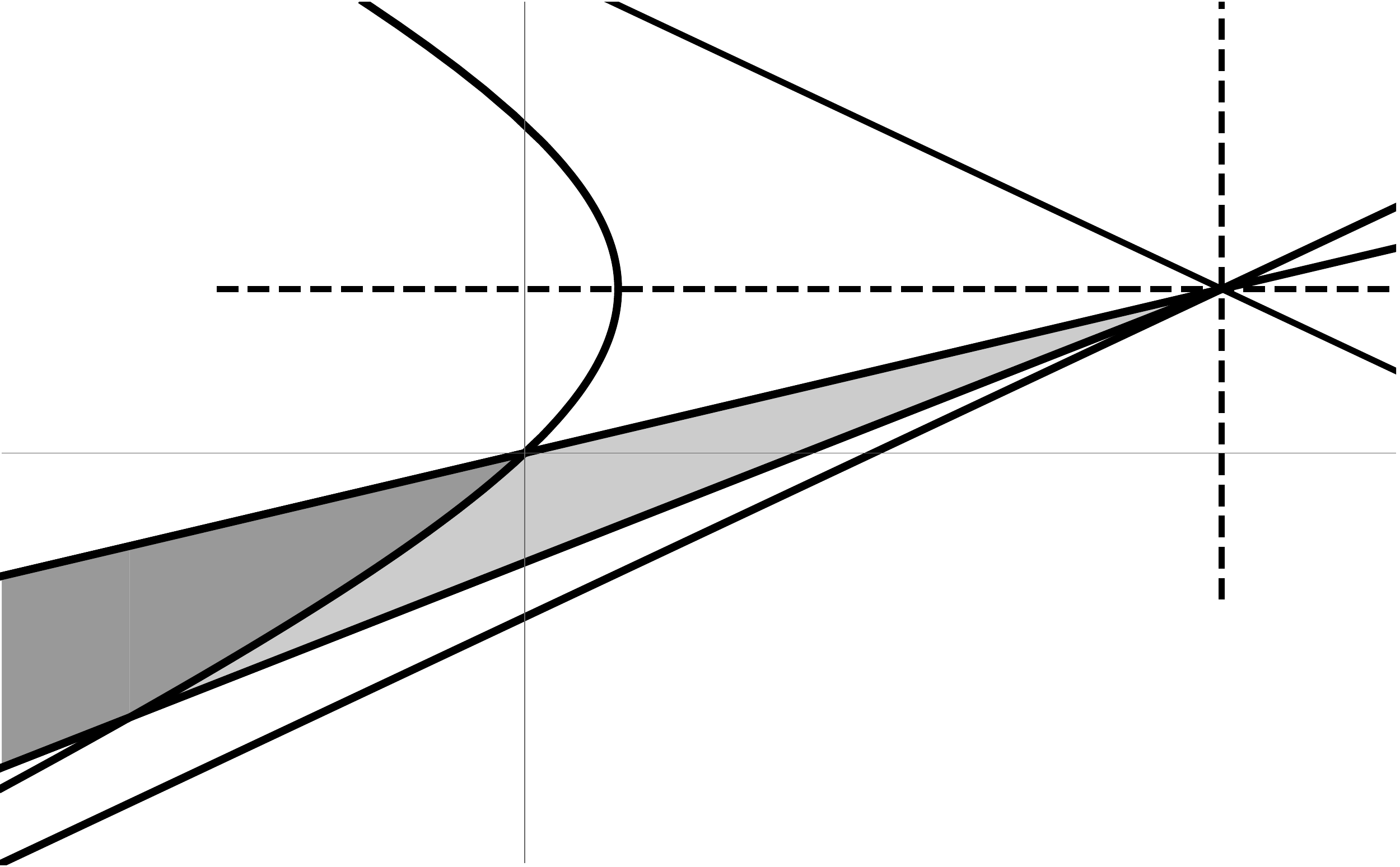}}
\put(-10,0.25){$d=4,\;p=2$}\put(-3,0.25){$d=4,\;p=6/5$}
\put(-0.25,1.5){$\gamma$}\put(-7.25,1.5){$\gamma$}
\put(-4.4,4){$\beta$}\put(-11.4,4){$\beta$}
\put(-6.5,1.92){\vector(1,0){6.5}}\put(-13.5,1.92){\vector(1,0){6.5}}
\put(-4.058,0){\vector(0,1){4.2}}\put(-11.058,0){\vector(0,1){4.2}}
\end{picture}
\caption{\label{Fig} \small\emph{In dimension $d=4$, the critical exponent $p=p_\star=d/(d-2)=2$ corresponds to the left figure, while $p=6/5$ is subcritical and corresponds to the right figure. The half cone of admissible regions of the parameters $(\beta,\gamma)$ appear in grey, with symmetry breaking in dark grey and symmetry in light grey (the symmetry region is bounded if and only if $p<p_\star$).}}
\end{center}
\end{figure}
\addtolength{\textwidth}{-30pt}

More is known. In the limit case $\beta=(d-2)\,\gamma/d$ corresponding to $p=p_\star=d/(d-2)$, we learn from~\cite{Catrina2001} that the optimal constant is achieved among radial functions, but that optimal functions exist only for $\beta\ge0$. The limit case as $\beta\to(\gamma-2)_+$, that is, $p\to1_+$, gives rise to a family of logarithmic Hardy inequalities which has been studied in~\cite{DDFT,DE2010}. We refer to~\cite{DEL-2015,Dolbeault2017} for more details on earlier contributions and to~\cite{1703} for a general overview. Stability results \emph{\`a la} Bianchi-Egnell, with no constructive estimates, appeared in~\cite{wei2021stability} in the critical case of~\eqref{CKN}, as an extension of the results of~\cite{MR1124290,figalli2020sharp}. So far there are no constructive stability results for~\eqref{CKN}.

\medskip Exactly as in the case of Inequalities~\eqref{GNS}, it is interesting to consider on $\R^d$ the nonlinear flow defined by
\be{FDE1}
\frac{\partial u}{\partial t}+|x|^\gamma\,\nabla\(|x|^{-\beta}\,u\,\nabla u^{m-1}\)=0\,,\quad u(t=0,\cdot)=u_0\,,
\ee
which generalizes~\eqref{FD} to $(\beta,\gamma)\neq(0,0)$ for $m\in[m_1,1)$ with a generalized critical exponent defined by $m_1:=1-\frac{2+\beta-\gamma}{2\,(d-\gamma)}$. The computation
\be{DE}
\frac \rd{\dt}\ird{u\,|x|^{-\gamma}}=0\,,\,\,\frac \rd{\dt}\ird{u^m\,|x|^{-\gamma}}=\frac{m^2}{1-m}\ird{u\,\left|\nabla u^{m-1}\right|^2\,|x|^{-\beta}}
\ee
enters in the formalism of the \emph{generalized R\'enyi entropy powers}. With $f=u^{m-1/2}$ and $p=1/(2\,m-1)$ so that $u=f^{2p}$, we notice that
\begin{multline*}
\textstyle\ird{u\,|x|^{-\gamma}}=\nrm f{2p,\gamma}^{2p}\,,\quad\ird{u^m\,|x|^{-\gamma}}=\nrm f{p+1,\gamma}^{p+1}\\
\textstyle\mbox{and}\quad\ird{u\,|\nabla u^{m-1/2}|^2\,|x|^{-\beta}}=\frac{4\,(1-m)^2}{(2\,m-1)^2}\,\nrm{\nabla f}{2,\beta}^2\,.
\end{multline*}
As in the non-weighted case, $m=m_1$ corresponds to $p=p_\star$. Taking into account~\eqref{CKN} and~\eqref{DE}, any solution of~\eqref{FDE1} satisfies
\[
\frac \rd{\dt}\ird{u^m\,|x|^{-\gamma}}\le C\(\ird{u^m\,|x|^{-\gamma}}\)^{-\frac{1-\theta}{\theta(p+1)}}
\]
for some numerical constant $C$ involving~$\C_{\beta,\gamma,p}$. Altogether, this proves a growth estimate similar to~\eqref{Ch2:GrowthEntropyGNS}. As for~\eqref{GNS}, this estimate is in fact equivalent to~\eqref{CKN}.

Again, more readable estimates are achieved using self-similar variables and relative entropies. Equation~\eqref{FDE1} can be rewritten in these variables~as
\be{FDE-CKN}
|x|^{-\gamma}\,\frac{\partial v}{\partial t}+\nabla\cdot\(|x|^{-\beta}\,v\,\nabla v^{m-1}\)=\sigma\,\nabla\cdot\(x\,|x|^{-\gamma}\,v\)\,.
\ee
The counterpart of Corollary~\ref{Cor:Improvedrate} is an improved decay rate of the free energy $\mathcal F$ now defined as
\begin{equation}\label{entropy.CKN}
\mathcal F[v]=\frac{2\,p}{1-p}\ird{\(v^\frac{p+1}{2\,p}-\mathsf g^{p+1}-\frac{p+1}{2\,p}\mathsf g^{1-p}\left(v-\mathsf g^{2p}\right)\)|x|^{-\gamma}}\,. 
\end{equation}
where $\mathsf g^{2p}$ is now a stationary solution to~\eqref{FDE-CKN} and $m=(p+1)/(2\,p)$. 
\begin{theorem}\label{Thm:Improvedrate} Let $d\ge1$. Assume that $(\beta,\gamma)\neq(0,0)$ satisfies
\be{condition.eta}
\gamma<d\,,\quad\gamma-2<\beta<\frac{d-2}d\,\gamma\quad\mbox{and}\quad\beta < \beta_{\rm FS}(\gamma)\,,
\ee
let $\alpha=1+(\beta-\gamma)/2$ and assume that $m\in[m_1,1)$. If $v$ solves~\eqref{FDE-CKN} with a nonnegative initial datum $v_0\in\mathrm L^{1,\gamma}(\R^d)$ such that $\ird{v_0(x)\,|x|^{-\gamma}}=\ird{\mathsf g(x)^{2p}\,|x|^{-\gamma}}$ and
\[
A[v_0]:=\sup_{R>0} R^{\frac{2+\beta-\gamma}{1-m}-(d-\gamma)}\,\int_{|x|>R}v_0(x)\,|x|^{-\gamma}\,\dx\,<\infty\,,
\]
then there are some $\zeta>0$ and some $T>0$ which depend explicitly only on $m$, $d$, $\nrm{v_0}{1,\gamma}$ and~$A[v_0]$ such that
\be{improved.rate.CKN}
\mathcal F[v(t,.)]\le\mathcal F[v_0]\,e^{-\,(4\,\alpha^2+\zeta)\,t}\quad\forall\,t\ge2\,T\,.
\ee
\end{theorem}
\noindent Under the restriction that $A[v_0]$ is finite (with a definition for $A$ which generalizes the one of Corollary~\ref{Cor:Improvedrate}), Inequality~\eqref{improved.rate.CKN} provides us with an improved rate of convergence since the optimal rate of convergence (without the restriction $A[v_0]<\infty$) for functions in $\mathrm{L}^{1,\gamma}(\R^d)$ is $4\,\alpha^2$, see Section~\ref{ssec:self.similar}. It is remarkable that no other condition is needed in the case $(\beta,\gamma)\neq(0,0)$, which is a major difference with Corollary~\ref{Cor:Improvedrate} where the conditions $\ird{v_0}=\ird{\mathcal B}$ and $\ird{x\,v_0}=0$ have to be assumed. Although somewhat hidden, these conditions are also present in Corollaries~\ref{Cor:Improvedrate2} and~\ref{Cor:Improvedrate3} as we use the \emph{optimized free energy functional} $\mathcal F_\star$ and the optimization with respect to $B\in\mathfrak B$ induces a similar normalization condition.

In Theorem~\ref{Thm:Improvedrate}, $T$ is a \emph{threshold time} which is similar to the threshold time in the non-weighted case and determines an \emph{asymptotic time layer} $[T,+\infty)$. We actually prove that
\begin{equation}\label{inq.to.prove}
\mathcal F[v(t,.)]\le\mathcal F[v(T,.)]\,e^{-\,(4\,\alpha^2+2\,\zeta)\,(t-T)}\quad\forall\,t\ge T\,.
\end{equation}
Compared with Corollary~\ref{Cor:Improvedrate}, we have no improved decay estimate on the \emph{initial time layer} $[0,T]$. The asymptotic decay rate of $\mathcal F[v(t,.)]$ as $t\to+\infty$ is known from~\cite{Bonforte201761,MR3579563}. Based on a fully quantitative regularity theory, the main progress here is that we give constructive estimates of~$T$ and an improved decay rate on $[T,+\infty)$.

\medskip\subsection{Simplifying assumptions and outline of the paper}\label{Sec:introSimplifications}

In this paper, we present some results in the spirit of~\cite{BDNS2021} and the key ideas of the proofs, with several simplifications:
\begin{itemize}
\item[--] we do not track the exact dependence of the constants on the parameters,
\item[--] we do not distinguish in the estimates quantities which depend on the relative entropy and those coming from $A[v]$ defined as in Theorem~\ref{Thm:Improvedrate}, on the basis of the following result:
\end{itemize}
\begin{lemma}\label{Lem:Entropy-MA}
Let $d\ge1$. Assume that the parameters $\beta$, $\gamma$ and $m$ are as in~Theorem~\ref{Thm:Improvedrate} and take $\sigma=2+\beta-\gamma$. There are two explicit positive numerical constants $c_1$ and~$c_2$ such that, for any nonnegative function $v\in\mathrm L^{1,\gamma}\big(\R^d,(1+|x|^\sigma)\,dx\big)$ such that $A[v]$ is finite, we have
\begin{align*}
&\ird{|x|^{\sigma-\gamma}\,v}\le\ird{v\,|x|^{-\gamma}}+c_1\,A[v]\,,\\
&\(\ird{v^m\,|x|^{-\gamma}}\)^{\kern-3pt\frac1m}\le c_2\(\ird {v\,|x|^{-\gamma}}\)^{\kern-3pt 1-(d-\gamma)\,\frac{1-m}{\sigma\,m}}\(\ird{|x|^{\sigma-\gamma}\,v}\)^{\kern-3pt(d-\gamma)\,\frac{1-m}{\sigma\,m}}.
\end{align*}
\end{lemma}
\noindent As a consequence, the relative entropy as defined in~\eqref{entropy.CKN} is controlled as soon as~$v$ has finite mass and $A[v]$ is finite. For a proof of Lemma~\ref{Lem:Entropy-MA}, it is easy to adapt the result of~\cite[Proposition~7.3]{BDNS2021} for the first inequality and use the Carlson-Levin estimate for the second inequality: see~\cite{Carlson1935,Levin48} and~\cite[Lemma~5]{Carrillo2019} for the proof of a similar result.

\medskip This paper is organized as follows. Section~\ref{Sec:RelativeError} is devoted to the regularization properties of the evolution equations~\eqref{FD} and~\eqref{FDE1}. On the basis of~\cite{Bonforte2019a}, the results on the \emph{relative uniform convergence} and on \emph{threshold time} $t_\star$ of~\cite{BDNS2021} for~\eqref{FD} are extended to~\eqref{FDE1}. With less details than in~\cite{BDNS2021}, all intermediate estimates are stated but only the differences with~\cite{BDNS2021} are emphasized: see Theorem~\ref{uniform.convergence} for the main result of the section. Entropy methods and \emph{improved entropy -- entropy production estimates} are applied in Section~\ref{Sec:Entropy-CKN} to prove Theorem~\ref{Thm:Improvedrate}. The fact that no additional constraint has to be imposed to get the improved decay rates for the solutions to~\eqref{FDE1}, a major difference with the standard fast diffusion equation~\eqref{FD}, is commented there. Section~\ref{Sec:StabGNS} is devoted to a summary of the strategy for proving the stability results for Gagliardo-Nirenberg-Sobolev inequalities, see Theorem~\ref{Thm:BDNS}, whose detailed proof can be found in~\cite{BDNS2021}. Our goal here is to explain that the improved decay rates can be extended to the initial time layer using a differential inequality based on the \emph{carr\'e du champ} method. Such an estimate is missing in the case of the evolution equation~\eqref{FDE1} associated with the Caffarelli-Kohn-Nirenberg inequalities, but a similar property is expected: this motivates the conjecture of Section~\ref{Sec:Conjecture-CKN}.

\section{A threshold time for the convergence in relative error}\label{Sec:RelativeError}

Equation~\eqref{FDE1} admits a family of self-similar solutions, that we call \emph{Barenblatt solutions} as a straightforward generalization of the non-weighted case, as in~\cite{Bonforte201761}. These solutions can be written as
\begin{equation}\label{Barenblatts}
B(t,x):=R(t)^{-d+\gamma}\,\mathsf g^{2p}\left(x/R(t)\right)
\end{equation}
where $\mathsf g$ is defined by~\eqref{Aubin-Talenti}, $p=1/(2\,m-1)$, $R(t)=c\,t^{1/\xi}$ for some constant $c>0$ that depends on $m$, $d$, $\beta$ and $\gamma$, and
\begin{equation}\label{xi}
\xi=2+\beta-\gamma -(d-\gamma)\,(1-m)\,.
\end{equation}
In order to fix notations, let us define
\be{Mstar}
\Mstar:=\ird{\mathsf g^{2p}}\,.
\ee
The purpose of this section is to prove that Barenblatt solutions attract all solutions of~\eqref{FDE1}. Notice that the case $(\beta,\gamma)=(0,0)$ is covered.

\medskip\subsection{Convergence in relative error}

The basin of attraction of the set of Barenblatt solutions in the strong topology of the \emph{uniform convergence in relative error} has been defined in~\cite{Vazquez2003}, and characterized in~\cite{Bonforte2019a}. It is a key point to estimate the stabilization rates, see~\cite{Blanchet2009,Bonforte201761}. The result goes as follows.
\begin{theorem}\label{uniform.convergence}
Let $d\ge2$, $m\in[m_1,1)$ and assume that $\beta$ and $\gamma$ satisfy~\eqref{parameters}. Let~$u$ be a solution to~\eqref{FDE1} corresponding to a nonnegative initial datum $u_0 \in \LL^{1,\gamma}(\R^d)$ such that $\ird{u_0\,|x|^{-\gamma}}=\Mstar$ and
\[
A[u_0]=\sup_{R>0} R^{\frac{2+\beta-\gamma}{1-m}-(d-\gamma)}\,\int_{|x|>R}u_0(x)\,|x|^{-\gamma}\,\dx<\infty\,.
\]
Then there exists an explicit $\varepsilon_\star$ such that for any $\varepsilon\in(0, \varepsilon_\star)$
\begin{equation}\label{relative.error}
\sup_{x\in\R^d}\Big|\frac{u(t,x)}{B(t,x)}-1\Big|\le \varepsilon\,\quad \forall\,t\ge t_\star:=\c\,\varepsilon^{-\a}\,.
\end{equation}
Here $\varepsilon_\star$ and $\a>0$ are numerical constants which depend only on $d$, $m$, $\beta$ and $\gamma$ while $\c$ depends also on $A[u_0]$.
\end{theorem}

\medskip\subsection{Global Harnack Principle}

\begin{proposition}\label{control.above}\cite[Theorem~2.1]{bonforte2020fine}
Under the assumptions of Theorem~\ref{uniform.convergence}, there exist positive constants $\overline{t}$ and $\overline{M}$ such that any solution $u$ to~\eqref{FDE1} satisfies
\begin{equation}\label{control.above.inq}
u(t,x) \le B_{\overline{M}}\big(t+\overline{t}, x\big)\,\quad \forall (t,x)\in\left[2\,\overline{t}\right.,\left. + \infty\right]\times\R^d\,. 
\end{equation}
where $B_{M}(t,x):= (M/\Mstar)^{\sigma/\xi}\,B\big(t, (M/ \Mstar)^{(1-m)/\xi}\,x\big)$.
\end{proposition}
\begin{proposition}\label{control.below}\cite[Theorem~3.1]{bonforte2020fine}
Under the assumptions of Theorem~\ref{uniform.convergence}, there exist positive constants $\underline{t}$ and $\underline{M}$ such that any solution $u$ to~\eqref{FDE1} satisfies
\begin{equation}\label{control.below.inq}
u(t,x) \ge B_{\underline{M}}\big(t-\underline{t}, x\big)\,\quad \forall (t,x)\in\left[2\,\underline{t}\right.,\left. + \infty\right]\times\R^d\,. 
\end{equation}
where $B_{M}(t,x)$ is as in Proposition~\ref{control.above}.
\end{proposition}
The proof of Propositions~\ref{control.above} and~\ref{control.below} can be found in~\cite{bonforte2020fine}, with another proof in~\cite[Propositions~4.6 and~4.7]{BDNS2021} which is better adapted to our purposes. Combining the results of Propositions~\ref{control.above} and~\ref{control.below} we obtain a precise control of the solution $u(t,x)$ which is called in the literature a~\emph{global Harnack principle}, see~\cite{MR2282669}. Let us stress that the quantities $\overline{t}$, $\underline{t}$, $\overline{M}$ and $\underline{M}$ can be explicitly computed and their value can be found in~\cite{bonforte2020fine}. We write below their dependencies with respect to the main parameters. In particular, we can chose $\overline{M}=\kappa_1\,\Mstar$ and $\underline{M}=\kappa_2\,\Mstar$ for some positive constants $\kappa_1$ and $\kappa_2$ which depend only on $d$, $m$, $\gamma$, and $\beta$.

\medskip\subsection{Convergence in relative error and the threshold time}

Let us define
\[
\overline{\varepsilon}:=\(\overline{M} /\Mstar\)^{\sigma/\xi}-1\,,\quad \underline{\varepsilon}:=1-\(\underline{M}/\Mstar\)^{\sigma/\xi}\,,\quad\mbox{and}\quad\varepsilon_{m}:=\min\left\{\tfrac12\,,\overline{\varepsilon}\,,\underline{\varepsilon}\right\} 
\]
where $\overline{M}$ and $\underline{M}$ are as in Propositions~\ref{control.above} and~\ref{control.below}. Integrating inequalities~\eqref{control.below.inq} and~\eqref{control.above.inq} over the whole space $\R^d$, we deduce that $\overline{M}>\Mstar$ and $\underline{M}<\Mstar$. As a consequence we obtain that $\overline{\varepsilon}$, $\underline{\varepsilon}$ and $\varepsilon_{m}$ are positive, and $\varepsilon_{m}$ depends only on $d$, $m$, $\beta$ and $\gamma$. 

\subsubsection*{The outer estimate} 
Here we compare a solution $u(t,x)$ with a Barenblatt profile with same mass $\Mstar$ as in~\cite[Section 4.4.1]{BDNS2021}, outside a large ball in $x$ and for large values of $t$. 
\begin{corollary}\label{control.tails}
Under the assumptions of Theorem~\ref{uniform.convergence} and for any $\varepsilon\in\left(0, \varepsilon_{m}\right)$ there exist $\rho(\varepsilon)$ and $T(\varepsilon)$ for which any solution $u$ to~\eqref{FDE1} satisfies
\begin{equation}\label{tail.control.inq}
\(1-\varepsilon\)\,B(t,x)\le u(t,x) \le \(1+\varepsilon\)\,B(t,x)\;\mbox{if}\;|x|\ge R(t)\,\rho(\varepsilon)\;\mbox{and}\;t\ge T(\varepsilon)\,.
\end{equation}
Furthermore, there exist positive constants $\overline{C}$ and $\underline{C}$ such that, for all $x\in\R^d$,
\begin{equation}\label{GHP.Barenblatt.delta}
\underline{C}\,B(t,x)\le u(t,x) \le \overline{C}\,B(t,x)\quad \forall\,t\ge 4\,T(\varepsilon)\,.
\end{equation}
\end{corollary}
An explicit expression of $T(\varepsilon)$ and $\rho(\varepsilon)$ can be computed and is not detailed here:  see~\cite[Section 4.4.1]{BDNS2021} for similar computations. We only remark that $\rho(\varepsilon)=O(1/\sqrt{\varepsilon}\,)$ and $T(\varepsilon)=\left(1+A\right)^{1-m}O(1/\varepsilon)$ as $\varepsilon\rightarrow 0$.
\begin{proof}
The proof is based on Propositions~\ref{control.above} and~\ref{control.below} and on the comparison (for $t$ large enough) of $B_M(t\pm c,x)$ with $B(t,x)$, where $c$ can be either $\overline{t}$ or $\underline{t}$ and $M$ either $\overline{M}$ or $\underline{M}$. We observe that the quotient $B_M(t\pm c ,x)/B(t,x)$ can be written as
\[
\frac{B_M(t\pm c ,x)}{B(t,x)}=\left(\frac{\lambda(t\pm c)}{\lambda(t)}\right)^{\gamma-d}\,\left(\frac{1+\lambda(t)^\sigma\,|x|^\sigma}{(M/\Mstar)^\frac{\sigma(1-m)}{\xi}+\lambda(t\pm c )^\sigma\,|x|^\sigma}\right)^\frac{1}{1-m}\,,
\] 
where $\lambda(t)=R(t)^{-1}$. Inequality~\eqref{tail.control.inq} follows from the fact that $\lambda(t\pm c )\sim\lambda(t)$ as $t\rightarrow \infty$. Inequality~\eqref{GHP.Barenblatt.delta} follows from a similar analysis performed directly on the same quotient $B_M(t\pm c ,x)/B(t, x)$.
\end{proof}

\subsubsection*{The inner estimate}
We consider what happens inside a ball as in~\cite[Section 4.4.2]{BDNS2021}.
\begin{corollary}\label{inner.estimate}
Under the assumptions of Theorem~\ref{uniform.convergence} and for any $\varepsilon\in \(0, \varepsilon_{m} \)$ and for any $t\ge 4\, T(\varepsilon)$, there exists a constant $\mathrm{K}>0$ and an exponent $\vartheta>0$ such that any solution $u$ to~\eqref{FDE1} satisfies
\begin{equation}\label{inner.esimate.inq}
\Big|\frac{u(t,x)}{B(t,x)}-1\Big|\le \frac{\mathrm{K}}{\varepsilon^\frac1{1-m}}\(\frac1{t}+\frac{\sqrt{\mathcal{F}[u_0]}}{R(t)}\)^\vartheta\,\quad\mbox{if}\quad |x|\le2\,\rho(\varepsilon)\,R(t)
\end{equation}
\end{corollary}
\noindent The constants $\mathrm{K}$ and $\vartheta$ are numerical constants, which depend only on $d$, $m$, $\beta$ and~$\gamma$. Their explicit values follow from the proof as in~\cite{BDNS2021}. 
\begin{proof}
The proof follows the proof of~\cite[Proposition~4.1]{BDNS2021} and further properties of parabolic regularity which are detailed in~\cite{BDNS2021}. We have to estimate the term $|u(t,x)/B(t,x)-1|$ in the ball of radius $2\,\rho(\varepsilon)\,R(t)$. This can be done by interpolating its $\LL^\infty$ norm between $\LL^p$ and $C^\mu$ semi-norms, defined on a bounded open domain~$\Omega$~as
\[
\lfloor u\rfloor_{C^\mu\left(\Omega\right)}:=\sup_{\substack{x,y\in\Omega\\x\neq y}}\frac{|u(x)-u(y)|}{|x-y|^\mu}\,,
\]
using the interpolation inequality (see~\cite[Section~3.1.4]{BDNS2021})
\begin{equation}\label{interpolation.inq}
\nrm u{\mathrm L^\infty(B_{R}(x))}\,\le\,C\left(\lfloor u\rfloor_{C^\mu(B_{2R}(x))}^{\frac{(d-\gamma)}{(d-\gamma)+p\,\mu}}\,\|u\|_{\mathrm L^{p,\gamma }(B_{2R}(x))}^{\frac{p\,\mu}{(d-\gamma)+p\,\mu}} + R^{-\frac{d-\gamma}p}\,\|u\|_{\mathrm L^{p,\gamma}(B_{2R}(x))}\right)\,.
\end{equation}
Here $C$ is a positive constant which depends on $d$, $\gamma$, $\mu$ and $p$ and \[
\textstyle\|u\|_{\mathrm L^{p,\gamma }(B_{R}(x_0))}:=\left(\int_{B_R(x_0)} |u|^p\,|x|^{-\gamma}\,\dx\right)^{1/p}\,.
\]

In order to estimate the $C^\mu$ semi-norm of $u(t,x)$ and $B(t,x)$ on the domain $|x|\le\, 2\,\rho(\varepsilon)R(t)$ with $t\ge 4\,T(\varepsilon)$, we have to deal with the time dependence of those functions and the domain itself. In order to simplify the analysis, we introduce the scaling $\hat{u}_{\tau, \kappa}(t,x):= \kappa^{\sigma/(1-m)}\,\tau^{(d-\gamma)/\xi}\,u\big(\tau\,t,\kappa\,\tau^{1/\xi}\,x\big)$ for some positive $\kappa$ and $\tau>0$. If $u$ is a solution to~\eqref{FDE1}, then so is $\hat{u}_{\tau, \kappa}$, for any $\tau, \kappa>0$. The Barenblatt profile is transformed by the previous scaling as $\hat{B}_{\tau, \kappa}(t,x)=B_{\kappa^{\xi/(1-m)}\,\Mstar}\big(t,x\big)$ where $B_{M}(t,x)$ is as in Proposition~\ref{control.above}. The advantage comes from the identity
\begin{equation}\label{inq.1}
 \frac{u(t,x)-B(t,x)}{B(t,x)}=\frac{\hat{u}_{t,1}\big(1, t^{-1/\xi}\,x\big)-\hat{B}_{t,1}\big(1, t^{-1/\xi}\,x\big)}{\hat{B}_{t,1}\big(1, t^{-1/\xi}\,x\big)}\,,
\end{equation}
as the domain $|x|\le 2\,\rho(\varepsilon)\,R(t) $ is included in $|y|\le 2\,Z\,\rho(\varepsilon)$ where $y=x\,t^{-1/\xi}$ for~$Z$ large enough. So, to estimate the $C^\mu$ semi-norm of the quotient $u(t,x)/B(t,x)$, it is enough to consider the right-hand side of~\eqref{inq.1} on a domain which is now independent of the time.

The denominator of the right-hand side of~\eqref{inq.1} can be estimated from below by a direct computation, while we use the parabolic regularity theory developed in~\cite{Bonforte2019a,Simonov2020} to bound the numerator $|\hat{u}_{t,1}(1, y)-\hat{B}_{t,1}(1, y)|$. We remark that the $C^\mu$-norm of $\hat{B}_{t,1}$ can be estimated by a direct computation, while for $\hat{u}_{t,1}$ we use the fact that it is a solution to a linear equation 
\begin{equation}\label{linear.equation}
\frac{\partial u}{\partial t}=|x|^\gamma\,\nabla\cdot\(|x|^{-\beta}\,A(t,x)\,\nabla u\)
\end{equation}
 where the coefficient $A(t,x)=m\,u^{m-1}(t,x)$. To obtain an estimate which is independent of $\varepsilon$, we estimate directly the $C^\mu$-norm of $\hat{u}_{t,1}$ on the whole space $\R^d$. We apply a standard trick in regularity theory: we cover $\R^d$ with subdomains of type $B_k(0) \setminus B_{k/2}(0)$. In order to estimate the norm of $\hat{u}_{t,1}$, we apply to the rescaled function $\hat{u}_{\tau, k}$ the identity
\[
\lfloor \hat{u}_{t, k}(1, \cdot)\rfloor_{C^\mu\left(B_1(0) \setminus B_{1/2}(0)\right)} = k^{\frac\sigma{1-m}+\mu}\,\lfloor \hat{u}_{t, 1}(1, \cdot)\rfloor_{C^\mu\left(B_k(0) \setminus B_{k/2}(0)\right)}\quad\forall\,t>0\,.
\]
By~\eqref{GHP.Barenblatt.delta}, the function $\hat{u}_{\tau, k}$ solves~\eqref{linear.equation} with a bounded and bounded away from zero coefficient $A(t,x)$. Therefore, by~\cite[Proposition~4.2]{Bonforte2019a}, there exists a constant $\overline{c}_1$, which depends only on $d$ and $n$, such that
\begin{equation*}\begin{split}
\lfloor \hat{u}_{t, k}(1, \cdot)\rfloor_{C^\mu\left(B_1(0) \setminus B_{1/2}(0)\right)} &\le \overline{c}_1\,\|\hat{u}_{t, k}\|_{\LL^\infty\left((\frac{1}{2},4)\times B_1(0) \setminus B_{1/2}(0)\right)}\\
&\le\overline{c}_1\,k^\frac\sigma{1-m}\,\|\hat{u}_{t, 1}\|_{\LL^\infty\left((\frac{1}{2},4)\times \R^d\right)}\,.
\end{split}
\end{equation*}
By setting $t_0=0$ and letting $R\rightarrow\infty$ in~\cite[Inequality~(2.1)]{Bonforte2019a}, we obtain the estimate $\|\hat{u}_{t, 1}\|_{\LL^\infty\left((\frac{1}{2},4)\times \R^d\right)}\le \overline{c}_2\,\Mstar$, for some positive constant $\overline{c}_2$ which depends only on $d$, $m$, $\beta$ and $\gamma$. Combining the above estimates, we find that
\begin{equation}\label{regularity.norm}\begin{split}
\lfloor \hat{u}_{t, 1}(1, \cdot)\rfloor_{C^\mu\left(\R^d\right)}&\le \lfloor \hat{u}_{t, 1}(1, \cdot)\rfloor_{C^\mu\left(B_1(0)\right)} + \sum_{k=1}^\infty \lfloor \hat{u}_{t, 1}(1, \cdot)\rfloor_{C^\mu\left(B_{2^{k+1}}(0)\setminus B_{2^k}(0)\right)} \\
&\le \overline{c}_1\,\overline{c}_2\,\frac{2^\mu}{2^\mu-1}\,\Mstar\,. 
\end{split}\end{equation}

In order to use Inequality~\eqref{interpolation.inq}, we need an estimate of $\nrm{u(t,x)-B(t,x)}{1,\gamma}$. To do so, we use the Csisz\'ar-Kullback-Pinsker inequality which allows us to control the evolution of $\nrm{u(t,x)-B(t+\overline{\tau},x)}{1,\gamma}$ where $\overline{\tau}>0$ is a time-shift needed for the definition of the relative entropy. Up to a scaling, $\overline{\tau}$ can be defined to be such that $B(\overline{\tau},x)=\mathsf{g}^{2p}(x)$. It is then convenient to use the triangle inequality as follows
\begin{equation}\label{inq.3}
\nrm{u(t,x)-B(t,x)}{1,\gamma}\le \nrm{u(t,x)-B(t+\overline{\tau},x)}{1,\gamma} + \nrm{B(t+\overline{\tau},x)-B(t,x)}{1,\gamma}\,.
\end{equation}
By a simple although lengthy computation, we have that 
\begin{equation}\label{time-shifted-norm}
\nrm{B(t+\overline{\tau},x)-B(t,x)}{1, \gamma}\le \frac{\overline{c}}{t}\quad\forall\,t\ge T(\varepsilon)\,.
\end{equation}
The Csisz\'ar-Kullback-Pinsker inequality as in~\cite[Lemma 2.12]{BDNS2021} implies that there exists an explicit constant $\mathsf{C}$, which depends only on $d$, $m$, $\beta$ and $\gamma$, such that
\begin{equation}\label{CKP}
\nrm{u(t,x)-B(t+\overline{\tau},x)}{1,\gamma} \le \mathsf{C}\, \frac{\sqrt{\mathcal{F}[u_0]}}{R(t)}\,.
\end{equation}
Thanks to~\eqref{tail.control.inq} it is sufficient to estimate the right-hand side of~\eqref{inq.1} in the domain $|y| \le 2\,Z\,\rho(\varepsilon)$ for $Z$ large enough.
Combining \eqref{interpolation.inq} with $p=1$, applied to the difference $\hat{u}_{t,1}-\hat{B}_{t,1}$ together with the estimates \eqref{inq.3},~\eqref{time-shifted-norm} and~\eqref{CKP}, we obtain for all $|y| \le 2\,Z\,\rho(\varepsilon)$ and $t\ge T(\varepsilon)$
\[
\Big|\hat{u}_{t,1}\big(1, y\big)-\hat{B}_{t,1}\big(1, y\big)\Big|\le \max\{\overline{c},\mathsf{C}\}\(\frac1{t}+\frac{\sqrt{\mathcal{F}[u_0]}}{R(t)}\)^\vartheta
\]
which allows to estimate the numerator of the right-hand side in~\eqref{inq.1}, and allows to conclude the proof of~\eqref{inner.esimate.inq} by estimating the denominator of the right-hand side in~\eqref{inq.1} by a direct computation.
\end{proof}

\medskip\subsection{Proof of Theorem~\texorpdfstring{\ref{uniform.convergence}}{7}}\label{ssec,pf.thm6}
\begin{proof}
From Corollary~\ref{control.tails} we deduce that
\begin{equation}\label{relative.error.control}
\Big|\frac{u(t,x)}{B(t,x)}-1\Big|<\varepsilon
\end{equation}
for $t \ge T(\varepsilon)$ and $|x|>\rho(\varepsilon)\,R(t)$, where $T(\varepsilon)=(1+A)^{1-m}\,O(1/\varepsilon)$. From Corollary~\ref{inner.estimate}, we deduce that inequality~\eqref{relative.error.control} holds if $t>4\,T(\varepsilon)$, $|x|\le 2\,\rho(\varepsilon)\,R(t)$ and $t>0$ is such that 
\[
\frac{\mathrm{K}}{\varepsilon^\frac1{1-m}}\(\frac1{t}+\frac{\sqrt{\mathcal{F}[u_0]}}{R(t)}\)^\vartheta\,<\varepsilon\,.
\]
Since $R(t)\le \left(C t\right)^\frac{1}{\xi}$, for some positive constant $C=C(d,m, \gamma, \beta)$, the last condition is satisfied if 
\begin{equation}\label{time.1}
t\ge \frac{1+\mathcal{F}[u_0]^\frac{\xi}{2}}{\varepsilon^{\a}}\quad\mbox{where}\quad\a:=\frac{\vartheta}{\xi}\,\frac{2-m}{1-m}
\end{equation}
Combining the above estimate~\eqref{time.1} with $T(\varepsilon)=(1+A)^{1-m}\,O(1/\varepsilon)$, by elementary computations, one finds that there exist a computable constant $\mathsf{C}$ which depends on $d$, $m$, $\gamma$, $\beta$, $\mathcal{F}[u_0]$ and $A[u_0]$ for which~\eqref{relative.error} holds for $t\ge\,\mathsf{C}(A[u_0], \mathcal{F}[u_0])\,\varepsilon^{-\a}$. The dependence on $\mathcal{F}[u_0]$ is eliminated using Lemma~\ref{Lem:Entropy-MA}, although more accurate estimates are obtained if the dependence on $\mathcal F[u_0]$ is kept as in~\cite{BDNS2021}. This completes the proof of Theorem~\ref{uniform.convergence}.
\end{proof}

\section{Improved entropy -- entropy production estimates}\label{Sec:Entropy-CKN}

We prove Theorem~\ref{Thm:Improvedrate} in Section~\ref{Sec:Theorem5} using an artificial dimension, entropy methods on the time interval $(t_\star,+\infty)$ with $t_\star$ given by Theorem~\ref{uniform.convergence} and spectral gap estimates, which are exposed respectively in Sections~\ref{ssub:artificial.dimension},~\ref{ssec:self.similar} and~\ref{ssec:gap}.

\medskip\subsection{An artificial dimension}\label{ssub:artificial.dimension}

Inequality~\eqref{CKN} can be recast as an interpolation inequality with same weight in all integrals which, in terms of scaling properties, amounts to introduce an \emph{artificial dimension}. To a function $f\in\mathrm H^p_{\beta,\gamma}(\R^d)$, let us associate the function $F\in\mathrm H^p_{\nu,\nu}(\R^d)$ with $\nu:=d-n<0$ such that
\be{ChangeOfDimension}
f(x)=F\(|x|^{\alpha-1}\,x\)\quad\forall\,x\in\R^d\,,
\ee
where
\[
\alpha=1+\frac{\beta-\gamma}2\quad\mbox{and}\quad n=2\,\frac{d-\gamma}{\beta+2-\gamma}\,.
\]
Notice that $p_\star=n/(n-2)$. In spherical coordinates, with $r=|x|$ and $\omega=x/r$ for any $x\in\R^d\setminus\{0\}$, let us define the derivation operator
\[
\DD U:=\(\alpha\,\frac{\partial U}{\partial r},\frac1r\,\nabla_{\kern-2pt\omega} U\)\,.
\]
With $\alpha>0$ and $p\in(1,p_\star]$, we can rewrite~\eqref{CKN} as
\be{CKN1}
\nrm U{2p,\nu}\le\mathsf K_{\alpha,n,p}\,\nrm{\DD U}{2,\nu}^\theta\,\nrm U{p+1,\nu}^{1-\theta}\quad\forall\,U\in \mathrm H^p_{\nu,\nu}(\R^d)\,,
\ee
for some optimal constant $\mathsf K_{\alpha,n,p}$ which is explicitly related with the optimal constant in~\eqref{CKN}: see~\cite[Proposition~6]{MR3579563}. Inequality~\eqref{CKN1} can be interpreted as a Gagliardo-Nirenberg-Sobolev inequality in the artificial dimension~$n$. As $\alpha\neq1$ unless $\beta=\gamma$, notice that symmetry issues in~\eqref{CKN1} are in no way simpler than in~\eqref{CKN}. A remarkable point is that the Aubin-Talenti type function as defined by~\eqref{Aubin-Talenti} is transformed into the more standard function
\[
x\mapsto\(1+|x|^2\)^\frac1{1-p}\,.
\]
We refer to~\cite[Section~2.3]{MR3579563} and~\cite[Section~3.1]{DEL-2015} for further details.

Through the transformation $u(t,x)=U\(t,|x|^{\alpha-1}\,x\)$, a solution $u$ of~\eqref{FDE1} is transformed into a solution of
\be{FDE2}
\frac{\partial U}{\partial t}=\L U^m
\ee
where $\DDstar$ denotes the adjoint of $\DD$ on $\mathrm L^2(\R^d,|x|^{-\nu}\,dx)$ and $\L=-\,\mathsf D_\alpha^*\,\DD$ is an elliptic self-adjoint differential operator given in spherical coordinates~by
\[
\L u=\alpha^2\(u''+\frac{n-1}r\,u'\)+\frac1{r^2}\,\Delta_\omega\,u\,.
\]
Functions obtained from $B(t,x)$ by reduction to the artificial dimension $n$ are self-similar Barenblatt solutions of~\eqref{FDE2}. If $\mathcal U$ is a solution of~\eqref{FDE2} with initial datum $\mB$ given by~\eqref{BarenblattGNS}, then $\mathcal U(t,x)=R(t)^{-n}\,\mB\big(x/R(t)\big)$ if and only if
\[
\frac{\rd R}{\dt}=\alpha^2\,R^{\,n\,(1-m)-1}\,.
\]
If we additionally assume that $R(0)=1$, this can be solved as
\[
R(t):=(1+\alpha^2\,\xi\,t)^{1/\xi}
\]
and $\xi=n\,(m-m_c)$ and $m_c=(n-2)/n$ as in~\eqref{xi}. Up to a time shift, this definition of $R(t)$ generalizes the definition of Section~\ref{Sec:RelativeError} to the case $\alpha\neq1$. Other self-similar Barenblatt solutions of~\eqref{FDE2} have same scaling properties and initial data given by~$\mB$, up to a multiplication by a constant and a scaling.

\medskip\subsection{Flow and entropies in self-similar variables}\label{ssec:self.similar}

Self-similar solutions suggest to rewrite~\eqref{FDE2} in the corresponding scales using the \emph{self-similar change of variables}
\be{SelfSimilarChangeOfVariables}
U(t,x)=\frac{\muscal^d}{R(t)^d}\,V\(\frac12\,\log R(t),\frac{\muscal\,x}{R(t)}\)
\ee
where $\muscal^{n\,(m-m_c)}=(1-m)/(2\,m)$. Hence if $U$ solves~\eqref{FDE2}, then $V$ solves
\be{Fdr2}
\frac{\partial V}{\partial t}=\DDstar\cdot\(V\(\DD V^{m-1}-2\,x\)\)\,,\quad V(t=0,\cdot)=V_0
\ee
with nonnegative initial datum $V_0=\muscal^{-d}\,U_0(\cdot/\muscal)\in\mathrm L^1(\R^d)$. Using homogeneity and scaling properties of~\eqref{Fdr2}, there is no restriction to fix $\ird{V_0\,|x|^{-\nu}}=\Mstar$, with $\Mstar$ defined by~\eqref{Mstar}. As a consequence, we shall assume from now on that
\[\label{Mass}
\ird{V(t,\cdot)\,|x|^{-\nu}}=\Mstar\quad\forall\,t\ge0
\]
without loss of generality.

The \emph{free energy} (or \emph{relative entropy}) and the \emph{Fisher information} (or \emph{relative entropy production}) are defined respectively by
\[
\mathcal F[V]:=\frac1{m-1}\ird{\(V^m-\mB^m-m\,\mB^{m-1}\,(V-\mB)\)|x|^{-\nu}}
\]
and, as a generalization of definition~\eqref{Fisher} of the Fisher information,
\[
\mathcal I[V]:=\frac m{1-m}\ird{V\,\left|\DD V^{m-1}-\DD\mB^{m-1}\right|^2\,|x|^{-\nu}}\,.
\]
With $V=|f|^{2\,p}$, $p=1/(2\,m-1)$, Inequality~\eqref{CKN1} is equivalent to the \emph{entropy -- entropy production inequality}
\be{entropy.eep}
\mathcal I[V]\ge4\,\alpha^2\,\mathcal F[V]
\ee
\emph{in the symmetry range} for~\eqref{CKN}. If $V$ solves~\eqref{Fdr2}, it is a straightforward computation to check that
\be{EEP-F-I}
\frac \rd{\dt}\mathcal F[V(t,\cdot)]=-\,\mathcal I[V(t,\cdot)]
\ee
after one integration by parts (which has to be justified: see~\cite{DEL-JEPE}), and as a consequence, we obtain that
\begin{equation}\label{standard.dacaying}
\mathcal F[V(t,\cdot)]\le\mathcal F[V_0]\,e^{-\,4\,\alpha^2\,t}\quad\forall\,t\ge0\,.
\end{equation}

\medskip\subsection{An improved spectral gap}\label{ssec:gap}

We consider the \emph{linearized free energy} and the \emph{linearized Fisher information} given respectively by
\[\label{linearized.fisher}
\mathsf F[h]:=\frac m2\ird{|h|^2\,\mB^{2-m}\,|x|^{-\nu}}\quad\mbox{and}\quad\mathsf I[h]:=m\,(1-m)\ird{|\DD h|^2\,\mB\,|x|^{-\nu}}\,.
\]
These quadratic forms are obtained as
\[
\mathsf F[h]=\lim_{\varepsilon\to0}\varepsilon^{-2}\,\mathcal F\big[\mB+\varepsilon\,\mB^{2-m}\,h\big]\quad\mbox{and}\quad\mathsf I[h]=\lim_{\varepsilon\to0}\varepsilon^{-2}\,\mathcal I\big[\mB+\varepsilon\,\mB^{2-m}\,h\big]\,.
\]
The following result is taken from~\cite[Proposition~4]{MR3579563}.
\begin{proposition}\label{Prop:SpectralGap} Let $d\ge2$, $\alpha\in(0,+\infty)$, $\nu=d-n<0$ and $\delta=1/(1-m)\ge n$. Then the \emph{Hardy-Poincar\'e inequality}
\[\label{Hardy-Poincare1}
\ird{|\DD h|^2\,\mB\,|x|^{-\nu}}\ge\Lambda\ird{|h|^2\,\mB^{2-m}\,|x|^{-\nu}}
\]
holds for any $h\in\mathrm L^2(\R^d,\mB^{2-m}\,|x|^{-\nu}\,dx)$ such that $\ird{h\,\mB^{2-m}}=0$, with an optimal constant $\Lambda$ given by
\[\label{Eqn:SpectralGap}
\Lambda=\left\{\begin{array}{rl}
2\,\alpha^2\,(2\,\delta-n)\quad&\mbox{if}\quad0<\alpha^2\le\frac{(d-1)\,\delta^2}{n\,(2\,\delta-n)\,(\delta-1)}\,,\\[6pt]
2\,\alpha^2\,\delta\,\eta\quad&\mbox{if}\quad\alpha^2>\frac{(d-1)\,\delta^2}{n\,(2\,\delta-n)\,(\delta-1)}\,,
\end{array}
\right.
\]
where $\eta$ is given by
\[\label{Eqn:eta}
\eta=\sqrt{\tfrac{d-1}{\alpha^2}+\big(\tfrac{n-2}2\big)^2}-\tfrac{n-2}2=\tfrac2{2+\beta-\gamma}\sqrt{d-1+\big(\tfrac{d-2-\beta}2\big)^2}-\tfrac{d-2-\beta}{2+\beta-\gamma}\,.
\]
\end{proposition}

\medskip\subsection{Proof of Theorem~\texorpdfstring{\ref{Thm:Improvedrate}}{1.2}}\label{Sec:Theorem5}

The constant in the entropy -- entropy production inequality~\eqref{entropy.eep} can be improved for a solution to~\eqref{Fdr2} if time is large enough and the initial datum $v_0$ is such that $A[v_0] < \infty$.
\begin{lemma}\label{improved.eep.lem} Assume that $(\beta,\gamma)\neq(0,0)$ satisfies~\eqref{condition.eta}. Let $\alpha$ and $\nu$ be as in Section~\ref{ssub:artificial.dimension} and define $\zeta:=\left(2\,(1-m)\,\Lambda-4\,\alpha^2\right)/4>0$ with $\Lambda$ as in Proposition~\ref{Prop:SpectralGap}. If $V$ is a solution to~\eqref{Fdr2} with an initial datum $V_0$ such that
\[
\sup_{R>0}R^{\frac{n(m-m_c)}{1-m}}\int_{|x|>R} V_0(x)\,|x|^{-\nu} \,\dx <\infty
\]
where $m_c=(n-2)/n$ as in~\eqref{xi}, then there exists $T>0$ such that
\begin{equation}\label{improved.eep}
\left( 4\,\alpha^2+ 2\,\zeta \right) \mathcal{F}[V(t,\cdot)] \le\,\mathcal I[V(t,\cdot)]\quad \forall\,t\ge T\,.
\end{equation}
\end{lemma}
\begin{proof} By applying the change of variables~\eqref{ChangeOfDimension} and~\eqref{SelfSimilarChangeOfVariables}, solutions to~\eqref{FDE-CKN} are tranformed in solutions to~\eqref{Fdr2} and in particular the shifted time-dependent Barenblatt profile $B(t+\overline{\tau},x)$ is transformed into the stationary solution~$\mB$ defined in~\eqref{BarenblattGNS}. As in Section~\ref{Sec:RelativeError}, $\overline{\tau}>0$ is such that $B(\overline{\tau},x)=\mathsf{g}^{2p}(x)$ and $B$ is as in~\eqref{Barenblatts}. Since
\[
\left\|\frac{B(t,x)}{B(t+\overline{\tau}, x)}\right\|_{\mathrm{L}^\infty(\R^d)}\le c_1\quad\mbox{and}\quad\left\|\frac{B(t+\overline{\tau}, x)-B(t,x)}{B(t+\overline{\tau}, x)}\right\|_{\mathrm{L}^\infty(\R^d)}\le \frac{c_2}{t}
\]
for $t>0$ large enough as a consequence of Theorem~\ref{uniform.convergence}, we find that for any $\varepsilon>0$ small enough there exists an explicit $T=T(\varepsilon)>0$ such that
\[\label{relative.error1}
\sup_{x\in\R^d}\left|\frac{V(t,x)}{\mB(x)}-1\right|\le \varepsilon\quad\forall\,t\ge T\,.
\]
According to~\cite[Lemma~18]{Bonforte201761}, one can deduce that
\[
\mathsf I[h(t,\cdot)]\le \frac{(1+\varepsilon)^{3-2m}}{(1-\varepsilon)}\,\mathcal I[V(t,\cdot)] +\,\varepsilon\,s_\varepsilon\,\mathsf F[h(t,\cdot)]\,,
\]
where $h(t,\cdot)=V(t,\cdot)\,\mB^{m-2}-\mB^{m-1}$ and $s_\varepsilon$ is a positive function of $\varepsilon$ such that $\lim_{\varepsilon\to0}s_\varepsilon>0$. As in the proof of~\cite[Lemma~19]{Bonforte201761}, using Proposition~\ref{Prop:SpectralGap} and the fact that $(1+\varepsilon)^{m-2}\le \mathcal F[V(t,\cdot)] / \mathsf F[h(t,\cdot)]\le (1-\varepsilon)^{m-2}$ by~\cite[Lemma~15]{Bonforte201761}, we obtain 
\[
\big(2\,(1-m)\,\Lambda -\,\rho_\varepsilon\,\varepsilon\big)\,\mathcal{F}[V(t,\cdot)] \le\,\mathcal I[V(t,\cdot)]\,,
\]
where $\rho_\varepsilon$ is bounded and stays bounded as $\varepsilon\rightarrow0$. A discussion has to be made depending on the cases in Proposition~\ref{Prop:SpectralGap}. When $2\,(1-m)\,\Lambda = 4\,\alpha^2\,(1-m)\,(2\,\delta - n)$, it follows from $m>m_1$ that $(1-m)\,(2\,\delta - n)>1$ and we find $2\,(1-m)\,\Lambda -\,\rho_\varepsilon\,\varepsilon > 4\,\alpha^2 + \zeta$ for $\varepsilon$ small enough. When $2\,(1-m)\,\Lambda = 4\,\alpha^2\,\eta$, the result follows because the condition $\eta>1$ is equivalent to~\eqref{condition.eta}.
\end{proof}

\begin{proof}[Proof of Theorem~\ref{Thm:Improvedrate}]
As a consequence of~\eqref{improved.eep} and~\eqref{EEP-F-I}, we have
\[
\frac \rd{\dt}\mathcal{F}[V(t,\cdot)]\le -\,(4\,\alpha^2+\zeta)\,\mathcal{F}[V(t,\cdot)]\,.
\]
Inequalities~\eqref{improved.rate.CKN} and~\eqref{inq.to.prove} are consequences of Gr\"{o}nwall's lemma.
\end{proof}

\section{A stability result for GNS}\label{Sec:StabGNS}

In this section, we deal with the non-weighted case $(\beta,\gamma)=(0,0)$. Detailed proof are given in~\cite{BDNS2021}. Here we simply outline the main steps of the proofs and emphasize the differences with the weighted case studied in Section~\ref{Sec:Entropy-CKN}.

\medskip\subsection{A quotient estimate}\label{sec:quotient-estimate}

For any function $v\neq\mB$ such that $\ird v=\Mstar$, with~$\Mstar$ defined by~\eqref{Mstar}, which is smooth enough and sufficiently decaying at infinity let us consider the quotient
\[\label{ch2:Q}
\mathcal Q[v]:=\frac{\mathcal I[v]}{\mathcal F[v]}
\]
where $\mathcal F[v]$ and $\mathcal I[v]$ are defined respectively by~\eqref{Entropy} and~\eqref{Fisher}.
\begin{lemma}\label{Lem:Q} Assume that $v$ solves~\eqref{FDr}. Then we have
\[\label{EqQ}
\frac \rd{\dt}\mathcal Q[v(t,\cdot)]\le\mathcal Q\,(\mathcal Q-4)\quad\forall\,t\ge0\,.
\]
\end{lemma}
Equation~\eqref{FDr} corresponds to the fast diffusion equation without weights, \emph{i.e.}, \hbox{$\beta=\gamma=0$} and the result follows from
\[\label{Eqn:Fisher}
\frac \rd{\dt}\mathcal I[v(t,\cdot)]\le-\,4\,\mathcal I[v(t,\cdot)]\,,
\]
which itself arises from the \emph{carr\'e du champ} method adapted to nonlinear flows. See for instance~\cite{DEL-JEPE} and references therein. With $(\beta,\gamma)\neq(0,0)$, such an estimate is so far formal.

\medskip\subsection{Scheme of the proof of Theorem~\texorpdfstring{\ref{Thm:BDNS}}{1.1}}
In the absence of weights the result follows by considering the improved spectral gap obtained for the flow~\eqref{FDr}. We shall only sketch the main steps of the proof and the interested reader may find the whole proof in~\cite[Chp. 5]{BDNS2021}. We first consider the subcritical case $1<p<p^\star$.

\smallskip
\noindent\textit{Step 0: normalization.} We notice that the deficit functional $\delta[\cdot]$ defined in~\eqref{deficit} is invariant by translations so that, without loss of generality, we assume that $|f|^{2p}$ has zero center of mass. As well, there is no harm to assume $\ird{|f|^{2p}}=\ird{\mathsf g^{2p}}$ where $\mathsf{g}$ is as in~\eqref{Aubin-TalentiGNS}, since the general case can be recover by scalings, see~\cite[Chapter~5]{BDNS2021}. We learn from~\cite[Lemma 1.12]{BDNS2021} that
\[
\frac{p+1}{p-1}\,\delta[f]=\mathcal I[v]-4\,\mathcal{F}[v]
\]
with the notation $v=|f|^{2p}$ and $p=1/(2\,m-1)$. Notice that translations are not allowed if $(\beta,\gamma)\neq(0,0)$, but that $\ird{x\,|f|^{2p}}=0$ is not required in that case.

\smallskip
\noindent\textit{Step 1: improved entropy -- entropy production inequality in the asymptotic time layer.} Let $v(t)$ be the solution the Cauchy problem~\eqref{FDr} with initial datum $v(0)=|f|^{2p}$. By Theorem~\ref{uniform.convergence} we know that, for $\varepsilon>0$ small enough, the relative error satisfies $|(v(t,\cdot)-\mathsf{g}^{2p})/\mathsf{g}^{2p}|<\varepsilon$ for any $t>t_\star$ where $t_\star$ is as in~\eqref{relative.error}. So, we are in the position of using Lemma~\ref{improved.eep.lem}: under the center of mass condition, the improved entropy -- entropy production inequality
\begin{equation}\label{improved.intial}
\mathcal I[v(t,\cdot)]\ge (4+\zeta)\,\mathcal{F}[v(t,\cdot)] \quad \forall\,t\ge t_\star
\end{equation}
holds with $\zeta=2\,d\,(m-m_1)$. The computation of the precise value of $\zeta$ can be found in~\cite[Proposition 2.10]{BDNS2021}. The additional constraint $\ird{x\,v_0}=0$ is needed to have $\zeta>0$.

\smallskip
\noindent\textit{Step 2: improved entropy -- entropy production inequality in the initial time layer.}
By integrating the differential inequality of Lemma~\ref{Lem:Q} backwards on $[0,t_\star]$, under the assumption that $Q(v(t_\star))\ge 4+ \zeta$, one finds 
\[
\mathcal I[v(t,\cdot)]\ge (4+\mu)\,\mathcal{F}[v(t,\cdot)]\quad\forall\,t\in[0,t_\star]\quad\mbox{where}\quad\mu=\frac{4\,\zeta\,e^{-4\,t_\star}}{4+\zeta-\zeta\,e^{-4\,t_\star}}\,.
\]
As a consequence, the improvement obtained in the asymptotic time layer $[t_\star,+\infty)$ can be transferred to the initial time layer $[0,t_\star]$ and up to the initial datum, with a smaller improvement of the constant. By multiplying the improved inequality by $4/(4+\mu)$, we obtain 
\[
\mathcal I[v]-4\,\mathcal{F}[v] \ge \frac{\mu}{4+\mu}\,\mathcal I[v]
\]
with $v=v(t,\cdot)$ for any $t\ge0$ and, as a special case, for $v=v_0$. Notice that the constant $\mu/(4+\mu)$ can be estimated explicitly since the dependence of $t_\star$ is given by~\eqref{relative.error}.

\smallskip
\noindent\textit{Step 3: rescaling and proof of Inequality~\eqref{Ineq:StabFisher}.} So far we have proven, with the above notation, that
\[
\frac{p+1}{p-1}\,\delta[f] \ge \mathcal{C}\,\mathcal I[v]\,,
\]
where $\mathcal{C}=4/(4+\mu)$.
To obtain inequality~\eqref{Ineq:StabFisher} we simply observe that 
\[
\mathcal I[v]\ge \inf_{\varphi\in\mathfrak M}\ird{\big|(p-1)\,\nabla f+f^p\,\nabla\varphi^{1-p}\big|^2}
\]
where the infimum is taken on the manifold of all optimal functions for~\eqref{GNS}. To obtain the result in its general form it is enough to rescale and go back to the original variables.

\smallskip
\noindent\textit{Step 4: the critical case $p=p^\star$.}
Due to the lack of an improved spectral gap in Step~1 under the previous scheme, with mass and center of mass constraints, a second moment constraint is also needed. This amounts to change the evolution equation to a new one which allows to control the evolution of the second moment as well: in practice we need to rescale Equation~\eqref{FDr} as done in~\cite[Chapter~5]{BDNS2021}. Entropy and entropy production should be optimized, \emph{i.e.}, considered with respect to the \emph{best matching Barenblatt} profile which is $\mathsf{g}^{2p}$ up to a time-dependent rescaling.This rescaling amounts to an additional time shift. Explicit estimates of the time shift requires some more work based on a system of ODEs and a phase portrait analysis. For more details see~\cite[Chapter~6]{BDNS2021}.

\medskip\subsection{Proof of Corollaries~\texorpdfstring{\ref{Cor:Improvedrate}}{1} and~\texorpdfstring{\ref{Cor:Improvedrate2}}{2}}
\begin{proof}[Proof of Corollary~\ref{Cor:Improvedrate}]
The entropy of the initial datum $\mathcal{F}[v_0]$ is finite under the current assumptions as a consequence of Lemma~\ref{Lem:Entropy-MA}. The proof of inequality~\eqref{improved.rate.GNS} follows from identity~\eqref{EEP-F-I} by combining a Gr\"{o}nwall argument with inequality~\eqref{improved.intial}. In this way we obtain
\begin{equation}\label{inq.ee}
\mathcal{F}[v(t,\cdot)]\le \mathcal{F}[v_0]\,e^{-\,(4+\zeta)\,t}\quad\forall\,t\ge0\,.
\end{equation}
\end{proof}
\begin{proof}[Proof of Corollary~\ref{Cor:Improvedrate2}]
Let us consider the sub-critical case $m_1<m<1$. We can obtain inequality~\eqref{standard.dacaying} from~\eqref{inq.ee} as follows. First, we can rescale the initial datum in such a way that $\mathcal{F}[v_0]=\mathcal{F}_\star[v_0]$. Then it is enough to take the infimum over all the Barenblatt profiles in the left-hand side of~\eqref{inq.ee}.

In the critical case $m=m_1$, we have to deal with an additional time-dependent scaling in order to control the evolution of the second moment. For more details see~\cite[Chapter~5]{BDNS2021}.
\end{proof}

\section{A conjecture on stability for Caffarelli-Kohn-Nirenberg inequalities}\label{Sec:Conjecture-CKN}

In Gagliardo-Nirenberg-Sobolev inequalities~\eqref{GNS} and in Caffarelli-Kohn-Nirenberg inequalities~\eqref{CKN}, the invariance under scalings plays an important role, as well as in the corresponding fast diffusion equations~\eqref{FD} and~\eqref{FDE1}. This explains why entropy methods are so efficient to get sharp results for the best constants as discussed in~\cite{DEL-JEPE}. A key feature is the \emph{carr\'e du champ} method, which has been rigorously implemented for~\eqref{FD} in~\cite{Carrillo2001}, in the context of parabolic equations. The regularity and decay estimates needed to justify the \emph{carr\'e du champ} method in the context of the fast diffusion flow associated with~\eqref{CKN} are so far missing, although some partial estimates are known from~\cite{DEL-JEPE,dolbeault2022parabolic}. This is why the symmetry results in~\cite{DEL-2015} were based on an elliptic version of the method, which formally also applies to parabolic equation~\eqref{FDE1}. Proving that all necessary integrations by parts can be justified would establish the following conjecture:
\begin{center}\emph{For some $\zeta>0$, Inequality~\eqref{improved.rate.CKN} holds for any $t\ge0$.}\end{center}
In other words, this means that in Theorem~\ref{Thm:Improvedrate}, one can take $T=0$. Indeed, the \emph{carr\'e du champ} estimate would allow us to extend the estimate on the asymptotic time layer $(T,+\infty)$ to the initial time layer $(0,T)$ and find a smaller but still constructive value for $\zeta$ depending only on the initial datum, by the same scheme as in~\cite{BDNS2021}. A straightforward consequence would be an improved entropy -- entropy production inequality that would provide us with a \emph{stability result with a constructive estimate} for Caffarelli-Kohn-Nirenberg inequalities~\eqref{CKN} similar to the result of Theorem~\ref{Thm:BDNS}. With this method, the stability would be measured by a relative Fisher information as for~\eqref{GNS}. Such a result is to be expected as, in the critical case of~\eqref{CKN}, a stability result \emph{without} constructive estimate has already been established by F.~Seuffert in~\cite{MR3695890}.

\section*{Acknowledgments}
\small M.B.~has been supported by the Projects \hbox{MTM2017-85757-P} and PID2020-113596GB-I00 (Ministry of Science and Innovation, Spain) and the Spanish Ministry of Science and Innovation, through the ``Severo Ochoa Programme for Centres of Excellence in R\&D'' (CEX2019-000904-S) and by the E.U.~H2020 MSCA programme, grant agreement 777822. J.D.~has been supported by the Project EFI (ANR-17-CE40-0030) of the French National Research Agency (ANR). N.S.~has been supported by the Spanish Ministry of Science and Innovation, through the FPI-grant BES-2015-072962, associated to the project MTM2014-52240-P (Spain) and by the E.U. H2020 MSCA programme, grant agreement 777822, by the Project EFI (ANR-17-CE40-0030) of the French National Research Agency (ANR), by the DIM Math-Innov of the Region \^Ile-de-France and by the ``Fondation Math\'ematique Jacques Hadamard". The authors thank an anonymous referee for his detailed reading, comments and suggestions.
\\
\copyright~2022 by the authors. This paper may be reproduced, in its entirety, for non-commercial purposes.
\bibliographystyle{aims}
\bibliography{BDNS-CKN.bib}
\begin{flushright}\emph{\small\today}\end{flushright}\bigskip
\end{document}